\documentclass[12pt]{article}
\usepackage{amsmath,amsfonts,amssymb,amsthm,amstext,color,graphicx,geometry}
\usepackage[backref]{hyperref}
\hypersetup{pdftitle={Critical variable exponent flows},
colorlinks=true,citecolor=blue,linkcolor=red}
\newcommand{\mres}{\mathbin{\vrule height 1.6ex depth 0pt width
0.13ex\vrule height 0.13ex depth 0pt width 1.3ex}}
 \newcommand{\abs}[1]{\left\vert#1\right\vert}
 \newcommand{\norm}[1]{\left\Vert#1\right\Vert}

 \newtheorem{theorem}{Theorem}[section]
 \newtheorem{lemma}[theorem]{Lemma}
  \newtheorem{proposition}[theorem]{Proposition}
  \theoremstyle{definition}
  \newtheorem{definition}[theorem]{Definition}
   
 \newtheorem{remark1}[theorem]{Observation}
  \theoremstyle{remark}
 \newtheorem{remark}[theorem]{Remark}

 \numberwithin{equation}{section}
\definecolor{maroon}{rgb}{0.5,0,0}
\newcommand {\R} {\mathbb{R}}
\newcommand {\M} {\mathbb{M}}
\def\be{\begin{equation}}
\def\ee{\end{equation}}
\def\fr#1#2{\frac{\partial #1}{\partial #2}}
\DeclareMathOperator{\divo}{div}

\begin{document}
\title{On time adaptive critical variable exponent vectorial diffusion flows and their applications in image processing}

\author{V. B. Surya Prasath\footnote{Department of Computer Science, University of Missouri-Columbia, MO 65211 USA. E-mail:prasaths@missouri.edu}
\and
D. Vorotnikov\footnote{Corresponding author. CMUC, Department of Mathematics, University of Coimbra, 3001-501 Coimbra, Portugal. E-mail: mitvorot@mat.uc.pt}}

\date{}

\maketitle
\begin{abstract}
Variable exponent spaces have found interesting applications in real world problems. Recently, there have been considerable interest in utilizing variational and evolution problems based on variable exponents for imaging applications. The main classes of partial differential equations (PDEs) related to the variable exponents involve the $p(\cdot)$-Laplacian. In imaging applications, the variable exponent can approach the critical value $1$, and this poses unique challenges in proving existence of solutions, which have not been mastered earlier.  In this work, we develop some additional functional framework to study the time-dependent parabolic flows with critical variable exponents. Specifically, we consider bounded vectorial partial variation (BVPV) space and its variable exponent counterpart. We prove the existence of weak solutions of critical vectorial $p(t,x)$-Laplacian flow in our variable exponent space.  For non-time-dependent variable exponent based critical vectorial $p(x)$-Laplacian flow we obtain a semigroup solution. The results are new even in the scalar case. This is a theory-oriented draft, and the full paper will provide detailed experimental results on color image restoration using various example for the variable exponents and compare them traditional PDE based image processing procedures. Our results will indicate the applicability of variable exponent Laplacian flows in image processing in general and image restoration in particular. 
\end{abstract}
\section{Conventions}

Fix a measurable variable exponent $p :(0,T)\times \Omega \to [1,p_+]$. We look for the solutions $u:(0,T)\times \Omega\to \R^N$ to the following Neumann problem \be \label{E:mainpl1} \left\{
\begin{array}{ll}
\displaystyle \partial_t u(t,x)=\divo [|\nabla u(t,x)|^{p(t,x)-2}\nabla u(t,x)]+f(t,x,u(t,x)),\\ \fr {u} {\nu} (t,x)=0,\ x\in \partial \Omega,
\\ u(0,x)=u_0(x). \end{array}
\right. \ee

 \label{conv}
$\Omega\subset\R^n$ is the bounded domain with Lipschitz boundary, $\R^N$ is the codomain of vector-valued functions $u(t,x)$. 

$Q=(0,T)\times\Omega$

$Q_{t}=(0,t)\times\Omega$

$\M=\R^{N\times n}$ 

For a vector function $u(t,x)\in \R^{N}$, we write $\nabla u(t,x)\in \M$ for its Jacobi matrix. For a matrix function $A=A_{ij}(t,x)\in \M$, 
the divergence operator acts as $(\divo A)_i=\sum\limits_{j=1}^n
\fr{A_{ij}}{x_j}.$

The Frobenius norm of a matrix $A\in \M$ is denoted by $|A|=\sqrt{Tr(A^\top A)}$.  

The scalar product of vectors from $\R^N$ and the componentwise scalar product in $\M$ is denoted by a dot. 

$\Omega$ is bounded, so when we refer to Radon measures, we always mean finite Radon measures. 

The Lebesgue measure of a set is denote by $|\cdot|$

$(\cdot,\cdot)$ and $\|\cdot\|$ are the scalar product and Euclidean norm in $L^2(\Omega,\R^N)$ or in  $L^2(\Omega,\M)$.

$p_+':=p_+/p_+-1$

We recall \cite{ABM14,Ev10} that the \emph{subdifferential} of a functional $\Phi:H\to \R\cup \{+\infty\}$ is the multivalued map $$\partial \Phi:\mathrm{dom}\, \partial \Phi\subset H \multimap H,$$ \be \label{sdd}\partial\Phi(u) = \left\{w\in H: \Phi(u+h) \geq \Phi(u)+(w,h)_H,\ \forall h\in H \right\}.\ee The set $\mathrm{dom}\, \partial \Phi$ consists of those $u\in H$ for which the right-hand side of \eqref{sdd} is non-empty. When $\Phi$ is differentiable, $\partial \Phi$ is single-valued and coincides with the usual gradient $\mathrm{grad}_H\, \Phi:H \to H$. 

\section{The functional framework}

\subsection{Preliminaries}
We give here some basic definitions of the bounded variation and variable exponent spaces, for further details we refer to \cite{Diening11,HHL07,Leoni}.  Let
$L^0(\Omega)$ be the set of measurable scalar functions on $\Omega \subset \R^n$. 
Fix a function $p\in L^0(\Omega)$ (the variable exponent). 
Denote
\[p_- = ess \inf_{\Omega} p(x), \quad \text{and} \quad p_+ = ess \sup_{\Omega} p(x).\] 
For our needs it suffices to assume that \[1 \leq p_- \leq p_+ <+ \infty.\] 

Consider the modular $\rho(u) = \int_{\Omega}\abs{u(x)}^{p(x)}\,dx$. Then \begin{equation*}
L^{p(\cdot)}(\Omega) := \{u\in L^0(\Omega)|\exists \lambda>0:\ \rho(u/\lambda) <+\infty\}.
\end{equation*}
The Luxemburg-Nakano norm on $L^{p(\cdot)}$ is
$\norm{u}_{L^{p(\cdot)}}:= \inf\{\lambda > 0: \rho(u/\lambda)\leq1\}$. We recall that \be\label{e:p-p} L^{p_+}(\Omega)\subset L^{p(\cdot)}(\Omega)\subset L^{p_-}(\Omega),\ee and the embeddings are continuous. 

The Orlicz-Sobolev space is  \[ W^{1,p(\cdot)}(\Omega) := \{
u\in L^{p(\cdot)}(\Omega) : \nabla u \in L^{p(\cdot)}(\Omega)^n\}\] with $\norm{u}_{W^{1,p(\cdot)}}:=
\|u\|_{L^{p(\cdot)}} +\||\nabla u|\|_{L^{p(\cdot)}}$. Then \be\label{e:wp-p} W^{1,p_+}(\Omega)\subset W^{1,p(\cdot)}(\Omega)\subset W^{1,p_-}(\Omega),\ee and the embeddings are continuous. 

The bounded variation space is defined as follows: a function $u\in BV(\Omega)$ if $u\in L^{1}(\Omega)$, and if its total variation
$$TV(u):=\sup\left\{\int_{\Omega} u\, \divo \phi : \phi\in C^1_0(\Omega)^n, \norm{\phi}_{L^{\infty}(\Omega)^n}\leq
1\right\}$$ is finite; the norm in $BV(\Omega)$ is given by
\[\norm{u}_{BV}:=\|u\|_{L^{1}}+ TV(u).\]
The distributional gradient $\nabla u$ is a vector-valued Radon measure, and hence the total variation $|\nabla u|$ (in the measure-theoretic sense) of this measure is a Radon measure. It is called the total variation measure and is sometimes also denoted by $TV(u)$. One can check that $TV(u)(\Omega)=TV(u).$

 In the case when $p_-=1$, the Orlicz-Sobolev space defined above is not reflexive, and it is plausible to use the space $BV^{p(\cdot)}(\Omega)$ defined below in various applications.  
 
 \begin{definition} A measurable function $p:\Omega\to [1,+\infty)$ is \emph{$Y$-semicontinuous} provided $p(x_0)=1$ for any sequence $x_k\to x_0$ in $\Omega$ such that $p(x_k)\to 1$.
 
 \end{definition}
 
 Assume in addition that $p(x)$ is a $Y$-semicontinuous function. Let $$Y:= \left\{x\in\Omega : p(x)
=1\right\}$$ be the critical set where the exponent takes the value $1$.  It is obviously closed in the relative topology of $\Omega$.

We now put 
$BV^{p(\cdot)}(\Omega):=BV(\Omega)\cap W^{1,p(\cdot)}(\Omega\setminus Y)$, and
consider the nonnegative number
\[TV^{p(\cdot)}(u):=TV(u)(Y) + \int_{\Omega\setminus Y} \abs{\nabla u}^{p(x)}\,dx.\]
Then the norm in $BV^{p(\cdot)}$ is \[\norm{u}_{BV^{p(\cdot)}(\Omega)}:= \norm{u}_{L^{p(\cdot)}} +
\inf \left\{\lambda >0 : TV^{p(\cdot)}(u/\lambda)\leq
1\right\}.\] 

All the spaces discussed above are Banach ones. 

\begin{remark} It is usually assumed that the exponent $p(\cdot)$ is at least lower semicontinuous \cite{HHLT13} in the definition of $BV^{p(\cdot)}(\Omega)$. However, our weaker assumption of $Y$-semicontinuity is enough. \end{remark}

\begin{remark} If $u\in BV(\Omega)^N$ is a vector function, then we can set \be TV(u):=\sup\left\{\int_{\Omega} u\cdot \divo \Phi : \Phi\in C^1_0(\Omega;\M), \norm{\Phi}_{L^{\infty}(\Omega;\M)}\leq
1\right\}.\ee Moreover (see Section \ref{s:sob} for more general considerations), the components of the distributional Jacobi matrix $\nabla u$ are signed Radon measures on $\Omega$, and the  total variation (Radon) measure is determined by the formula \be TV(u)(E)=\sup\left\{\sum_k |\nabla u(E_k)|\right\},\ee where the supremum is taken over all Borel partitions $\{E_k\}$ of a Borel set $E\subset \Omega$.\end{remark}

\subsection{Sobolev-BVPV spaces} \label{s:sob}

The standard variable exponent function spaces (which were recalled in the previous section) are suitable for stationary problems. We now develop some additional functional framework in order to study parabolic variable exponent flows.

 Let $\mathcal M$ be the the continuous dual of the space $C_0(Q)^N$  of continuous vector functions on $Q$ that vanish on the boundary of the cylinder $Q$. By the Riesz duality \cite[Theorem B114]{Leoni}, it can be also viewed as the Banach space of $N$-vectorial signed Radon measures on $Q$ equipped with the total variation norm \be\|v\|_{\mathcal M}=|v|(Q),\ee where the Radon measure $|v|$ is the total variation of the measure $v$. We recall that \be |v|(E)=\sup\left\{\sum_k |v(E_k)|\right\},\ee
 where the supremum is taken over all Borel partitions $\{E_k\}$ of a Borel set $E\subset Q$. For $v\in \mathcal M$, we define the \emph{vectorial partial variation} of $v$ as  
\be\label{E:tvf1}VPV(v)=\sup\limits_{\Phi\in C_0^1(Q;\M):\, |\Phi|\leq 1}\left\langle v,\divo \Phi\right\rangle_{\mathcal M\times C_0(Q)^N}.\ee Observe that if $v\in L^1(0,T;W^1_1(\Omega))^N$, then $VPV(v)=\int_0^T\int_\Omega |\nabla v(t,x)|\,dx\,dt$. 

Define the ``bounded vectorial partial variation space'' $BVPV$ as $$\left\{ v\in \mathcal M\Big| \|v\|_{BVPV}:=\|v\|_{\mathcal M}+VPV(v)<+\infty\right\}.$$

The following proposition is straightforward by the Riesz duality:
\begin{proposition} \label{L:P3}  For any $v\in BVPV$, the components of its distributional Jacobi matrix $\nabla v$ are signed Radon measures on $Q$.\end{proposition} For any $v\in BVPV$ and a Borel set $E\subset Q$, consider the nonnegative scalar \be VPV(v)(E)=\sup\left\{\sum_k |\nabla v(E_k)|\right\},\ee where the supremum is taken over all Borel partitions $\{E_k\}$ of $E$. Then \cite[Proposition B75 and Theorem B114]{Leoni}and Proposition \ref{L:P3} imply
\begin{proposition} For any $v\in BVPV$, $VPV(v)$ is a Radon measure on $Q$, and $VPV(v)(Q)=VPV(v)$.\end{proposition}

For every open set $U\subset Q$ and $v\in BVPV$, applying \cite[Theorem B114]{Leoni} to the measure $\nabla v\mres U \in (C_0(U;\M))^*$, we infer that \be\label{E:tvf3}VPV(v)(U)=\sup\limits_{\Phi\in C_0^\infty(U;\M):\, |\Phi|\leq 1}\left\langle v,\divo \Phi\right\rangle_{\mathcal M(U)\times C_0(U)^N}.\ee

Owing to lower semicontinuity of suprema, we can derive from \eqref{E:tvf3} the following fact:
\begin{proposition} \label{L:twin} For any weakly-* converging (in $\mathcal{M}$) sequence $\{v_m\}\subset BVPV$, one has \be \label{E:twin1} VPV(v)\leq \liminf_{m\to +\infty}  VPV(v_m)\ee where $v=\lim v_m$. Moreover, \be \label{E:twinU} VPV(v)(U)\leq \liminf_{m\to +\infty}  VPV(v_m)(U)\ee for every open set $U\subset Q$. \end{proposition}

Then \eqref{E:twin1} yields \begin{proposition} \label{L:twba} $BVPV$ is a Banach space. \end{proposition}

We now fix a $Y$-semicontinuous scalar  function $p :Q \to [p_-,p_+]\subset [1,+\infty)$ (the variable exponent). Note that the critical set \be Y:=[p(t,x)=1]\ee is closed in the relative topology of $Q$. 

The evolutionary Orlicz-Sobolev space $W^{0,1}_{p(\cdot)}
(Q\backslash Y)^N$ is defined as \[ W^{0,1}_{p(\cdot)}
(Q\backslash Y)^N := \{
v\in L^{p(\cdot)}(Q\backslash Y)^N : \nabla v \in L^{p(\cdot)}(\Omega;\M)\}.\] It is easy to establish \begin{proposition} \label{p:p5b} The space $W^{0,1}_{p(\cdot)}
(Q\backslash Y)^N$ equipped with the norm $$\norm{v}_{W^{0,1}_{p(\cdot)}
(Q\backslash Y)^N}:=
\||v|\|_{L^{p(\cdot)}(Q\backslash Y)} 
+\||\nabla v|\|_{L^{p(\cdot)}(Q\backslash Y)}.$$ is a Banach space. It is reflexive provided $p_->1$. \end{proposition}

We define the subset $BVPV^{p(\cdot)}$ of $\mathcal{M}$ as follows: a vectorial Radon measure $v$ belongs to $BVPV^{p(\cdot)}$ provided $VPV(v)<\infty$, the restriction $v \mres (Q\backslash Y)$ is absolutely continuous with respect to the Lebesgue measure, and the Radon-Nikodym derivative \[v|_{Q\backslash Y}:=\frac{d (v \mres (Q\backslash Y))}{dx\otimes dt}\] belongs to $W^{0,1}_{p(\cdot)}
(Q\backslash Y)^N$. This definition is expressed by the following descriptive but sloppy formula: \be\label{bvpv} BVPV^{p(\cdot)}=BVPV \cap 
W^{0,1}_{p(\cdot)}
(Q\backslash Y)^N.\ee

Given a function $v\in BVPV^{p(\cdot)}$, we define the
 nonnegative number \be \label{vpvp}VPV^{p(\cdot)}(v)=VPV(v)(Y) +
  \int_{Q\backslash Y} |\nabla v(t,x)|^{p(t,x)}\,dx\,dt.\ee

\begin{proposition} $BVPV^{p(\cdot)}$ is a Banach space, being equipped with the norm
\be\label{norm1}\|v\|_{BVPV_{p(\cdot)}}:=|v|(Y)+\|v\|_{L^{p(\cdot)}(Q\backslash Y)^N} + 
\inf \left\{\lambda >0 : VPV^{p(\cdot)}(v/\lambda)\leq
1\right\} , \ee  or with the equivalent norm \be\label{norm2}\|v\|_{BVPV^{p(\cdot)}}:=\|v\|_{BVPV} + \|v\|_{W^{0,1}_{p(\cdot)}
(Q\backslash Y)^N}.\ee \end{proposition}

\begin{proof} It is clear that \eqref{norm2} is a norm on $BVPV^{p(\cdot)}$. Let $\{v_k\}$ be a Cauchy sequence in this norm. By Proposition \ref{L:twba}, it admits a limit $v$ in $BVPV$. Hence, $v_k \mres (Q\backslash Y)\to v \mres (Q\backslash Y)$ in total variation. By Proposition \ref{p:p5b}, $v_k|_{Q\backslash Y}$ has a limit in $W^{0,1}_{p(\cdot)}
(Q\backslash Y)^N$, which should coincide with $v|_{Q\backslash Y}$. Thus, $BVPV^{p(\cdot)}$ is a Banach space.

Observe now that $ VPV^{p(\cdot)}(v)<+\infty$  for $v\in BVPV^{p(\cdot)}.$
Hence, $$\inf \left\{\lambda >0 : VPV^{p(\cdot)}(v/\lambda)\right\}<+\infty$$ for all $v\in BVPV^{p(\cdot)}$. Furthermore, $\inf \left\{\lambda >0 : VPV^{p(\cdot)}(v/\lambda)\right\}$ is a seminorm on $BVPV^{p(\cdot)}$ by an argument similar to the proof of \cite[Theorem 2.1.7]{Diening11}. Consequently, \eqref{norm1} is a norm on $BVPV^{p(\cdot)}$.  

It remains to conclude that both norms are equivalent since, due to \eqref{e:p-p}, \begin{multline}\|v\|_{BVPV_{p(\cdot)}}\to 0\\ \iff |v_k|(Y)+VPV(v_k)(Y) +
  \int_{Q\backslash Y} |v_k(t,x)|^{p(t,x)} +|\nabla v_k(t,x)|^{p(t,x)}\,dx\,dt \to 0\\ \iff |v_k|(Q)+VPV(v_k)+
    \int_{Q\backslash Y} |v_k(t,x)|^{p(t,x)} +|\nabla v_k(t,x)|^{p(t,x)}\,dx\,dt\to 0\\ \iff \|v\|_{BVPV^{p(\cdot)}}\to 0.\end{multline}
    \end{proof}
    
\begin{remark}The space $BVPV^{p(\cdot)}$ is relevant for image processing applications since it allows for jumps on the critical set $Y$ even if this critical set is merely a null set in $Q$ (e.g., a hypersurface in $Q$). \end{remark}    
    
    For any function $v\in BVPV^{p(\cdot)}$, we define the measure $VPV^{p(\cdot)}(v)$ on $Q$ by letting \be VPV^{p(\cdot)}(v)=VPV(v)\mres Y+
 |\nabla v(t,x)|^{p(t,x)}\,dx\,dt\mres (Q \setminus Y).\ee The following observation is clear in view of \eqref{vpvp}: \begin{proposition} For any $v\in BVPV^{p(\cdot)}$, $VPV^{p(\cdot)}(v)$ is a Radon measure on $Q$, and $VPV(v)^{p(\cdot)}(Q)=VPV^{p(\cdot)}(v)$.\end{proposition}
The following lower-semicontinuity property of $VPV^{p(\cdot)}$ is crucial in the applications. 
    \begin{theorem} \label{L:twinp} For any open set $U\subset Q$, and for any weakly-* converging (in $\mathcal{M}$) sequence $\{v_m\}\subset BVPV^{p(\cdot)}$, $v=\lim v_m$, one has \be \label{E:twinpp} VPV^{p(\cdot)}(v)(U)\leq \liminf_{m\to +\infty}  VPV^{p(\cdot)}(v_m)(U). \ee 
    In particular, if $U=Q$ and the limit inferior is finite, then $v\in BVPV^{p(\cdot)}.$
    \end{theorem}
\begin{proof} Assume that $U=Q$. 
Mutatis mutandis, the same argument works for $U \neq Q$.

Let the limit inferior in \eqref{E:twinpp} be finite (otherwise the theorem holds trivially). By Proposition \ref{L:twin}, $v\in BVPV$. 

For any $\epsilon>0$, consider the sets  \be Y_\epsilon:=\left\{s_0\in Q, \inf_{s\in \partial Q\cup Y}|s_0-s|\leq \epsilon\right\},\ee which are closed in the relative topology of $Q$. Let $U_\epsilon:=\mathring{Y}_\epsilon$.

Then \be\label{e:au1} |\partial Y_\epsilon|= |\partial U_\epsilon|=0\ee (at least up to a countable set of $\epsilon$'s), and \be\label{e:au2} \inf_{(t,x)\not\in Y_\epsilon} p(t,x)>p_-.\ee Moreover, $\cap_{\epsilon>0}(Y_\epsilon \backslash Y)=\varnothing$, whence \be\label{e:au3} |Y_\epsilon \backslash Y|\to 0,\ee
\be\label{e:au4} \int_{U_\epsilon\backslash Y} |\nabla v(t,x)|^{p(t,x)}\,dx\,dt  \to 0\ee as $\epsilon \to 0$.

The spaces $W^{0,1}_{p(\cdot)}
(Q\backslash Y_\epsilon)^N$ are reflexive by Proposition \ref{p:p5b}. Consequently, passing to a subsequence if necessary, \be v_m|_{Q\backslash Y_\epsilon}\to v|_{Q\backslash Y_\epsilon}\ \mathrm{weakly}\ \mathrm{in}\ W^{0,1}_{p(\cdot)}
(Q\backslash Y_\epsilon)^N.\ee Hence, by \cite[Theorem 2.2.8]{Diening11},  \begin{multline}\label{e:modpp} \int_{Q\backslash Y_\epsilon} |\nabla v(t,x)|^{p(t,x)}\,dx\,dt\leq \liminf_{m\to +\infty}  \int_{Q\backslash Y_\epsilon} |\nabla v_m(t,x)|^{p(t,x)}\,dx\,dt\\ \leq \liminf_{m\to +\infty}  VPV^{p(\cdot)}(v_m). \end{multline}
Letting $\epsilon\to 0$, and remembering \eqref{e:au3}, by the Lebesgue monotone convergence theorem we infer 
\be\int_{Q\backslash Y} |\nabla v(t,x)|^{p(t,x)}\,dx\,dt\leq \liminf_{m\to +\infty}  \int_{Q\backslash Y} |\nabla v_m(t,x)|^{p(t,x)}\,dx\,dt.\ee Hence, $v\in BVPV^{p(\cdot)}.$ 

By Proposition \ref{L:twin}, \begin{multline} \label{E:twin4} VPV(v)(U_\epsilon)\leq \liminf_{m\to +\infty}  VPV(v_m)(U_\epsilon)\\=\liminf_{m\to +\infty}  VPV(v_m)(Y)+\liminf_{m\to +\infty}  \int_{U_\epsilon\backslash Y} |\nabla v_m(t,x)|\,dx\,dt\\ \leq \liminf_{m\to +\infty}  VPV(v_m)(Y)+\liminf_{m\to +\infty}  \int_{U_\epsilon\backslash Y} 
\left(1+|\nabla v_m(t,x)|^{p(t,x)}\right)\,dx\,dt.
\end{multline}

Employing \eqref{e:au1}, \eqref{e:au2}, \eqref{e:modpp} and \eqref{E:twin4}, we conclude that
\begin{multline}\label{E:twinpp1} VPV^{p(\cdot)}(v)= VPV(v)(Y)+\int_{Q\backslash Y} |\nabla v(t,x)|^{p(t,x)}\,dx\,dt\\\leq  VPV(v)(U_\epsilon) +\int_{Q\backslash Y_\epsilon} |\nabla v(t,x)|^{p(t,x)}\,dx\,dt+\int_{U_\epsilon\backslash Y} |\nabla v(t,x)|^{p(t,x)}\,dx\,dt\\ \leq \liminf_{m\to +\infty} VPV(v_m)(Y)+|U_\epsilon\backslash Y|+\liminf_{m\to +\infty} \int_{U_\epsilon\backslash Y} 
|\nabla v_m(t,x)|^{p(t,x)}\,dx\,dt \\ + \liminf_{m\to +\infty}  \int_{Q\backslash Y_\epsilon} |\nabla v_m(t,x)|^{p(t,x)}\,dx\,dt +\int_{U_\epsilon\backslash Y} |\nabla v(t,x)|^{p(t,x)}\,dx\,dt
\\ = |U_\epsilon\backslash Y|+\int_{U_\epsilon\backslash Y} |\nabla v(t,x)|^{p(t,x)}\,dx\,dt+\liminf_{m\to +\infty}  VPV^{p(\cdot)}(v_m), \end{multline}
and \eqref{E:twinpp} follows due to \eqref{e:au3} and \eqref{e:au4}.
  \end{proof}

\section{The critical vectorial $p(t,x)$-Laplacian flow}

Fix a $Y$-semicontinuous variable exponent $p :Q \to [1,p_+]$ with $Y=[p(t,x)=1]$ being the critical set in $Q$. We are interested in defining and studying the solutions $u:Q\to \R^N$, $Z\in Q\to \M$ to the following Neumann problem \be \label{E:mainpl} \left\{
\begin{array}{ll}
\displaystyle \partial_t u(t,x)=\divo Z(t,x)+f(t,x,u(t,x)),\\
Z(t,x)=|\nabla u(t,x)|^{p(t,x)-2}\nabla u(t,x),\\ Z(t,x)\nu(x) =0,\ x\in \partial \Omega,
\\ u(0,x)=u_0(x). \end{array}
\right. \ee 
The fidelity term is a measurable function $f:Q\times \R^N\to \R^N$, which is sublinear in the sense
\be\label{fid} |f(t,x,u)|\leq \tilde f(t,x)+C|u|\ee for some measurable function $\tilde f:Q\to \R$. 

\subsection{Weak solutions}

\begin{definition} \label{D:wwdefpx}  
Let $u_0\in L^2(\Omega)^N$, $\tilde f\in L^2(Q)$. A pair of functions $(u,Z)$ from the class \begin{multline} \label{E:uclass1} u\in C_w(0,T; L^2(\Omega)^N)\cap BVPV^{p(\cdot)}\cap W^1_{p_+'}(0,T; [L^2 \cap W^{1,p_+}(\Omega)^N]^*), \\ Z\in L^{p_+'}(Q;\M)\cap L^{\infty}(Y;\M),\ \|Z\|_{ L^\infty(Y;\M)}\leq 1,\end{multline} is called a weak solution
to problem \eqref{E:mainpl}  if \\ i) for all $v\in L^2\cap W^{1,p_+}(\Omega)^N$,
 \be\label{E:zp} \left\langle\frac {du} {dt},v \right\rangle+ \big(Z, \nabla v\big)=(f(\cdot,u),v) \ee a.e. on $(0,T)$;
 \\ ii)
for all $w\in L^{p_+}(0,T;W^{1,p_+}(\Omega))^N\cap W^{1,1}(0,T;L^2(\Omega))^N$ and any open set $U$, $Y\subset U\subset Q$, 
one has  \begin{multline} \label{E:wvep} 
\|u(T)-w(T)\|^2 +2\int_Q \partial_t w(t,x) \cdot u(t,x)\,dx\,dt+2\, VPV^{p(\cdot)}(u)(U)\\ \leq \|u_0-w(0)\|^2 +\|
w(T)\|^2-\|w(0)\|^2+2
\int_U Z(t,x)\cdot\nabla w(t,x) \,dx\,dt\\
+2\int_{Q\backslash U}|\nabla w(t,x)|^{p(t,x)-2}\nabla w(t,x) \cdot (\nabla w(t,x)-\nabla u(t,x))\,dx\,dt\\ +2\int_Q f(t,x,u(t,x))\cdot(u(t,x)-w(t,x))\,dx\,dt\;
\end{multline}
iii) the initial condition \be \label{E:ic} u(0)=u_0\ee  holds in the space $L^2(\Omega)^N$.  
\end{definition}
\begin{remark1} \label{R:1var}
Let us present a motivation for this definition of solution to \eqref{E:mainpl}. Further motivation is discussed in Observation \ref{r:ssol} below.

Consider a pair $(u,Z)$ of sufficiently regular functions satisfying \eqref{E:uclass1},\eqref{E:zp}, \eqref{E:wvep}, \eqref{E:ic}. Then we can rewrite \eqref{E:wvep} in the form \begin{multline} \label{E:wvepr0} 
2\int_Q \partial_t u(t,x) \cdot (u(t,x)-w(t,x))\,dx\,dt+2\int_U |\nabla u(t,x)|^{p(t,x)}\,dx\,dt\\ \leq 2
\int_U Z(t,x)\cdot\nabla w(t,x) \,dx\,dt\\
+2\int_{Q\backslash U}|\nabla w(t,x)|^{p(t,x)-2}\nabla w(t,x) \cdot (\nabla w(t,x)-\nabla u(t,x))\,dx\,dt\\ +2\int_Q f(t,x,u(t,x))\cdot(u(t,x)-w(t,x))\,dx\,dt.
\end{multline} Due to \eqref{E:zp} it reduces to \begin{multline} \label{E:wvepr} 
\int_U |\nabla u(t,x)|^{p(t,x)} \,dx\,dt\leq \int_Q Z(t,x)\cdot(\nabla u(t,x)-\nabla w(t,x)) \,dx\,dt\\+
\int_U Z(t,x)\cdot\nabla w(t,x) \,dx\,dt\\
+\int_{Q\backslash U}|\nabla w(t,x)|^{p(t,x)-2}\nabla w(t,x) \cdot (\nabla w(t,x)-\nabla u(t,x))\,dx\,dt.
\end{multline}  Testing \eqref{E:wvepr} by $w\equiv u$, we get 
 \begin{multline}\int_Y |\nabla u(t,x)| \,dx\,dt\leq  \int_U |\nabla u(t,x)|^{p(t,x)} \,dx\,dt \\ \leq
\int_U Z(t,x)\cdot\nabla u(t,x) \,dx\,dt, \end{multline} Letting $U=U_\epsilon$ where $U_\epsilon$ was defined during the proof of Theorem \ref{L:twinp}, and allowing $\epsilon \to 0$, we derive 
 \be \int_Y |\nabla u(t,x)| \,dx\,dt\leq 
\int_Y Z(t,x)\cdot\nabla u(t,x) \,dx\,dt, \ee
On the other hand, \be   |\nabla u|\geq Z\cdot\nabla u \ee a.e. in $Y$, whence\be \label{E:wve4p}  |\nabla u|=Z\cdot\nabla u, \ee and  \be  \label{E:wve5p}  Z=\frac {\nabla u}{|\nabla u|}=|\nabla u|^{p(\cdot)-2}\nabla u\ee a.e. in $Y$. Rigorously speaking, if $|Y\cap [\nabla u=0]|>0$, then \eqref{E:wve5p} merely holds in the following subdifferential sense  \be  \label{E:wve5psd}  Z\in \partial \psi(\nabla u)\ee a.e. in $Y$, where $ \partial \psi$ is the subdifferential (cf. Section \ref{conv}) of the function $\psi: \M\to \R,$ $\psi(A)= |A|$.

Substituting expression \eqref{E:wve4p} to the left-hand side of  \eqref{E:wvepr}, we derive that \begin{multline}
\int_{Q\backslash U}(Z-|\nabla w(t,x)|^{p(t,x)-2}\nabla w(t,x)) \cdot (\nabla u(t,x)-\nabla w(t,x))\,dx\,dt\\ \geq \int_{U\setminus Y} |\nabla u(t,x)|^{p(t,x)} \,dx\,dt - \int_{U\setminus Y} Z(t,x)\cdot\nabla u(t,x) \,dx\,dt.
\end{multline} Letting $U=U_\epsilon$, $\epsilon \to 0$, we get  \be \label{E:mbineq} 
\int_{Q\backslash Y}(Z-|\nabla w(t,x)|^{p(t,x)-2}\nabla w(t,x)) \cdot (\nabla u(t,x)-\nabla w(t,x))\,dx\,dt\geq 0. 
\ee By a Minty argument, this yields \be  \label{E:wve6p}  Z=|\nabla u|^{p(\cdot)-2}\nabla u\ee a.e. in $Q\backslash Y$. 

Integrating be parts in \eqref{E:zp}, we infer \be \label{e:fin1p} (\partial_t u - \divo Z-f(\cdot,u), v)+ \int_{\partial \Omega} v\cdot Z \nu\, d \mathcal{H}^{n-1}=0.\ee
Testing \eqref{e:fin1p} by any smooth $v$  compactly supported in $\Omega$, we deduce 
the first equation in \eqref{E:mainpl}. Consequently, both integrals in   \eqref{e:fin1p} are identically zero. By arbitrariness of $v$,\ $Z\nu=0$ \ $\mathcal{H}^{n-1}$-a.e. in $\partial \Omega$. 
\end{remark1}

\begin{theorem} \label{mthm} If $u_0\in L^2(\Omega)^N$, $\tilde f\in L^2(Q)$, and the function $v\mapsto v\cdot f(t,x,v)$ is concave for a.a. $(t,x)\in Q$. Then there exists a weak solution to \eqref{E:mainpl}. 
\end{theorem}

\subsection{Disgression: on solvability of some nonlinear Cauchy problems in Hilbert triples} 

Assume that there are two Hilbert spaces, $X\subset Y,$ with
continuous embedding operator $i:X\to Y$, and $i(X)$ is dense in
$Y$. The adjoint operator $i^*:Y^*\to X^*$ is continuous and,
since $i(X)$ is dense in $Y$, one-to-one. Since $i$ is one-to-one,
$i^*(Y^*)$ is dense in $X^*$, and one may identify $Y^*$ with a
dense subspace of $X^*$. Due to the Riesz representation theorem,
one may also identify $Y$ with $Y^*$. We arrive at the chain of
inclusions:
\begin{equation} X\subset Y\equiv Y^*\subset X^*.\end{equation}
Both embeddings here are dense and continuous. Observe that in
this situation, for $f\in Y, u\in X$, their scalar product in $Y$
coincides with the value of the functional $f$ from $X^*$ on the
element $u\in X$:
\begin{equation} \label{ttt} (f,u)_Y=\langle f,u \rangle. \end{equation}
Such triples $(X,Y, X^*)$ are called \emph{Hilbert triples} (sometimes also referred to as Gelfand or Lions triples), see, e.g., \cite{Te79,ZV08} for more details. 

\begin{lemma} \label{l:l1} Let $$X\subset Y\subset X^*$$ be a Hilbert triple. Let $\mathcal{A}:X\to X^*$ be a linear continuous operator such that $$\langle \mathcal A u , u \rangle \geq \alpha \|u\|_X^2$$ for all $u\in X$ and some common $\alpha>0$. Let $V$ be a Banach space such that $$X\subset V\subset Y$$ where the first embedding is compact and the second is continuous. Assume that both $X$ and $V$ are separable. Assume moreover that \be\label{e.thet}\|u\|_V\leq \|u\|_X^\theta \|u\|_Y^{1-\theta}\ee for some $\theta\in (0,1)$. Let $$\mathcal{Q}_L, \mathcal{Q}_M:(0,T)\times V\to X^*$$ be Carath\'eodory functions, i.e., $\mathcal{Q}_L(t,\cdot), \mathcal{Q}_M(t,\cdot)$ are continuous for a.a. $t\in (0,T)$ and $\mathcal{Q}_L(\cdot,u), \mathcal{Q}_M(\cdot,u)$ are measurable for all $u\in V$. Assume that \be\label{e:q1}\|\mathcal{Q}_L (t,u)\|_{X^*} \leq C(F(t)+\|u\|_V), \ee \be\label{e:q2}\|\mathcal{Q}_M (t,u)\|_{X^*} \leq C(F(t)+\|u\|_V^{m}),\  \langle \mathcal{Q}_M (t,u) , u \rangle \leq 0. \ee for all $u \in X$ and a.a. $t\in (0,T)$, where $F\in L^2(0,T)$ is independent of $u$, and $m\geq 1$, $m\theta <1$. Then the Cauchy problem \begin{equation}\label{E:acp1} u^\prime(t) +\mathcal Au(t)=\mathcal Q_L(t,u(t))+\mathcal Q_M(t,u(t)), \quad  u|_{t=0}=u_0,
\end{equation} has a solution in the class \be L^{2}(0,T;X)\cap H^{1}(0,T;X^*)\cap L^{2/\theta}(0,T;V) \cap C([0,T];Y)\ee for every $u_0\in Y$. \end{lemma}

\begin{proof}

1. Consider the Banach space $$W:=L^{2}(0,T;X)\cap H^{1}(0,T;X^*)$$ with the natural intersection norm. By the Aubin-Lions-Simon lemma \cite{Si87}, the embedding $W\subset L^2(0,T;V)$ is compact.  Moreover, by \cite[Lemma 2.2.7 and Corollary 2.2.3]{ZV08} any function from $W$ may be considered as an element of $C([0,T];Y)$, and the embedding $W\subset C([0,T];Y)$ is continuous. Due to \eqref{e.thet}, \be\label{e:tet1} \|u\|_{L^{2/\theta}(0;T;V)}\leq \|u\|^\theta_{L^{2}(0,T;X)}\|u\|^{1-\theta}_{L^{\infty}(0,T;Y)},\ee therefore the embedding $W\subset L^{2/\theta}(0;T;V)$ is also continuous. By \cite[Corollary 6]{Si87}, the embedding $W\subset L^{2m}(0;T;V)$ is compact (as $2m<2/\theta$).
The Nemytskii operators $$Q_L(\cdot,u): L^2(0,T; V)\to L^2(0,T; X^*),$$ $$Q_M(\cdot,u): L^{2m}(0,T; V)\to L^2(0,T; X^*)$$ are bounded and continuous by the Lucchetti-Patrone theorem \cite{LP80,mt5}. Hence, the Nemytskii operators $$Q_L(\cdot,u), Q_M(\cdot,u): W\to L^2(0,T; X^*)$$ are continuous and compact. 

2. Consider the family of Cauchy problems \begin{equation}\label{E:acp1l} u^\prime +\mathcal Au=\lambda(\mathcal Q_L(t,u)+\mathcal Q_M(t,u)), \quad  u|_{t=0}=u_0,
\end{equation} where $\lambda \in [0,1]$ is a parameter. We are going to prove that the solutions of \eqref{E:acp1l} are a priori bounded in the space $W$, uniformly in $\lambda$. Taking the scalar product in $Y$ of \eqref{E:acp1l} and $u$ a.e. in $(0,T)$, we infer \be\frac 1 2 \frac {d} {dt} \|u\|_Y^2+\alpha \|u\|_X^2\leq \lambda \langle \mathcal Q_L(t,u)+Q_M(t,u), u \rangle\leq C(F(t)+\|u\|_V)\|u\|_X.\ee Therefore, \be\frac {d} {dt} \|u\|_Y^2+\alpha \|u\|_X^2\leq  C(F^2(t)+\|u\|^2_V)\leq  C(F^2(t)+\|u\|^{2\theta}_X\|u\|^{2(1-\theta)}_Y),\ee and by Young's inequality \be\label{e:512}\frac {d} {dt} \|u\|_Y^2+\frac \alpha 2 \|u\|_X^2\leq  C(F^2(t)+\|u\|^2_V)\leq  C(F^2(t)+\|u\|^{2}_Y).\ee By Gr\"onwall's lemma, \be \label{e:513}\|u\|_{L^{2}(0,T;X)}+\|u\|_{L^{\infty}(0,T;Y)}\leq C,\ee hence, by \eqref{e:tet1}, \be\|u\|_{L^{2m}(0,T;V)}\leq C.\ee Consequently, \eqref{E:acp1l} together with \eqref{e:q1}, \eqref{e:q2} yields \be\|u^\prime\|_{L^{2}(0,T;X^*)}\leq C.\ee

3.  By \cite[Lemma 3.1.3]{ZV08} and Banach's continuous inverse theorem, the linear operator \[\tilde{\mathcal A}: W\to Y \times L^2(0,T;X^*), \quad \tilde{\mathcal A}(u)=(u|_{t=0}, u^\prime +\mathcal Au)\] is continuously invertible. Now \eqref{E:acp1l} can be rewritten as \begin{equation}\label{E:acp3l} u=\lambda \tilde{\mathcal A}^{-1}(0,\mathcal Q_L(\cdot,u)+\mathcal Q_M(\cdot,u)). 
\end{equation} The operator $\tilde{\mathcal Q}(u):=\tilde{\mathcal A}^{-1}(0,\mathcal Q_L(\cdot,u)+\mathcal Q_M(\cdot,u))$ is continuous and compact in $W$, and the solutions of \eqref{E:acp3l} are uniformly bounded in $W$. By Schaefer's fixed point theorem \cite{Ev10}, $\tilde{\mathcal Q}$ has a fixed point, which is obviously a solution of \eqref{E:acp1}.
\end{proof}

\subsection{Proof of Theorem \ref{mthm}.}

In order to prove the theorem we first need to study the following auxiliary problem:  \begin{equation}
\label{e:maina}
\left\{
\begin{array}{ll}
\displaystyle
\partial_t u +\delta \mathcal  B u= \divo \left(\frac {|\nabla u|^{2p(\cdot)-2}\nabla u}{\epsilon+|\nabla u|^{p(\cdot)}}\right)+f(\cdot,u),
\\
u(0, x) = u_{0}(x).
\end{array}
\right.
\end{equation} Here $\epsilon, \delta$ are positive parameters; the operator $\mathcal  B$ and the boundary conditions are to be specified below.

Consider the Hilbert triple $$(H^r(\Omega)^N, L^2(\Omega)^N, \left((H^r(\Omega))^*\right)^N)$$ where $$r>\left(1+\frac n 2\right)(p_+ +1)$$ is a fixed number.  Denote by $\mathcal  B$ the Riesz bijection between the spaces $(H^r)^N$ and $((H^r)^*)^N$. 

The weak form of \eqref{e:maina} (with a certain implicit boundary condition which is of no importance to us) is the following Cauchy problem, where the first equality is understood as an ODE in the space $((H^r)^*)^N$, whereas the second equality is in the sense of the space $(L^2)^N$:
\begin{equation}\label{E:apprp1} u^\prime +\delta \mathcal Bu=\mathcal Q(t,u)+f(t,\cdot,u), \quad  u|_{t=0}=u_0.
\end{equation}  
Here the operator $\mathcal Q:(0,T)\times C^1(\overline\Omega)^N\to \left((H^r(\Omega))^*\right)^N$ is determined by the duality $$\langle \mathcal Q(t,u), w \rangle= -\left(\frac {|\nabla u|^{2p(t,\cdot)-2}\nabla u}{\epsilon+|\nabla u|^{p(t,\cdot)}}, \nabla w\right), \quad \forall  w \in (H^r)^N,$$ and the prime stands for the time derivative.

\begin{lemma} \label{L:ap} Let $u_0\in L^2(\Omega)^N$, $\tilde f\in L^2(Q)$. The Cauchy problem \eqref{E:apprp1} admits a solution $u$ in the class \begin{equation}\label{e:skl} L^{2}(0,T;(H^r(\Omega))^N)\cap H^{1}(0,T;((H^r(\Omega))^*)^N)\cap C([0,T];(L^2(\Omega))^N). \end{equation} The solution satisfies the following inequality: \begin{multline} \label{E:wvepap} 
\|u(t_*)-w(t_*)\|^2 +2\int_{Q_{t_*}} \partial_t w(t,x) \cdot u(t,x)\,dx\,dt+2\, VPV^{p(\cdot)}(u)({Q_{t_*} \cap U})\\ \leq \|u_0-w(0)\|^2 +\|
w(t_*)\|^2-\|w(0)\|^2+2
\int_{Q_{t_*} \cap U} \frac {|\nabla u(t,x)|^{2p(t,x)-2}}{\epsilon+|\nabla u(t,x)|^{p(t,x)}}\nabla u(t,x)\cdot\nabla w(t,x) \,dx\,dt\\
+2\int_{Q_{t_*}\backslash U}\frac{|\nabla w(t,x)|^{2p(t,x)-2}}{\epsilon+|\nabla w(t,x)|^{p(t,x)}}\nabla w(t,x) \cdot (\nabla w(t,x)-\nabla u(t,x))\,dx\,dt\\ +2\int_{Q_{t_*} \cap U} f(t,x,u(t,x))\cdot(u(t,x)-w(t,x))\,dx\,dt\,\\ +2\epsilon |Q_{t_*} \cap U|+\frac \delta 2 \|w\|_{L^{2}(0,t_*;(H^r(\Omega))^N)}
\end{multline} for every $t_*\in [0,T]$, every open set $U$, $Y\subset U\subset Q$, and for every sufficiently regular test function $w:\overline Q \to \R^N$.\end{lemma}  \begin{proof}

1. We are going to apply Lemma \ref{l:l1} with $$X=H^r(\Omega)^N,\, Y=L^2(\Omega)^N,\, V=C^1(\overline\Omega)^N,\, \mathcal Q_M=\mathcal Q,\, \mathcal Q_L=f,$$ $$\mathcal A=\delta \mathcal B,\, m=p_+,\, \theta=
\frac 1 {p_+ +1}.$$ A classical interpolation inequality in  Sobolev Hilbert scale \cite{Kr72} gives \be\|u\|_{H^{r\theta}}\leq C\|u\|_{H^r}^\theta \|u\|_{L^2}^{1-\theta}.\ee Since $r\theta >1+\frac n 2$, we have by Sobolev embedding that \be\|u\|_{C^1(\overline\Omega)}\leq C \|u\|_{H^{r\theta}},\ee so \eqref{e.thet} is satisfied. Moreover, the embedding $H^r(\Omega)^N\subset C^1(\overline\Omega)^N$ is compact. The operators $\mathcal{Q},f:(0,T)\times C^1(\overline\Omega)^N\to \left((H^r(\Omega))^*\right)^N$ satisfy the Carath\'eodory condition. Moreover, $f$ satisfies condition \eqref{e:q1} since \be \|f(t,\cdot,u)\|_{\left((H^r(\Omega))^*\right)^N}\leq \|f(t,\cdot,u)\|_{L^2(\Omega)^N}\leq \|\tilde f(t)\|_{L^2(\Omega)}+C \|u\|_{C^1(\overline\Omega)^N},\ee and $\tilde f\in L^2(0,T;L^2(\Omega))$. Furthermore, $\mathcal Q$ satisfies \eqref{e:q2} since \begin{multline} \label{e:q2l} \|\mathcal{Q} (t,u)\|_{\left((H^r(\Omega))^*\right)^N} \leq \|\mathcal{Q} (t,u)\|_{\left((H^1(\Omega))^*\right)^N}\leq \int_\Omega\left|\frac {|\nabla u|^{2p(\cdot)-2}\nabla u}{\epsilon+|\nabla u|^{p(\cdot)}}\right|\\ \leq \int_\Omega\left(1+|\nabla u|^{p_+-1}\right)\leq C(1+\|u\|_{C^1(\overline\Omega)^N}^{p_+}).\end{multline}  Thus, the conditions of Lemma \ref{l:l1} are met, and the existence of solution $u$ follows immediately. 

2. We now fix a sufficiently smooth function $w:\overline Q\to \R$ and an open set $U$, $Y\subset U\subset Q$. Test \eqref{E:apprp1} with $2(u-w)(t)$, $0\leq t\leq t_*\leq T$, in the sense of $(((H^r(\Omega))^*)^N, H^r(\Omega)^N)$-duality, and integrate in time to obtain \begin{multline} \label{E:wvepap1} 
\|u(t_*)-w(t_*)\|^2 +2\int_{Q_{t_*}} \partial_t w(t,x) \cdot u(t,x)\,dx\,dt+2 \int_{Q_{t_*} \cap U} \frac{|\nabla u(t,x)|^{2p(t,x)}}{\epsilon+|\nabla u(t,x)|^{p(t,x)}}\,dx\,dt\\ \leq \|u_0-w(0)\|^2 +\|
w(t_*)\|^2-\|w(0)\|^2+2
\int_{Q_{t_*} \cap U} \frac {|\nabla u(t,x)|^{2p(t,x)-2}}{\epsilon+|\nabla u(t,x)|^{p(t,x)}}\nabla u(t,x)\cdot\nabla w(t,x) \,dx\,dt\\
+2\int_{Q_{t_*}\backslash U}\frac{|\nabla u(t,x)|^{2p(t,x)-2}}{\epsilon+|\nabla u(t,x)|^{p(t,x)}}\nabla u(t,x) \cdot (\nabla w(t,x)-\nabla u(t,x))\,dx\,dt\\ +2\int_{Q_{t_*}} f(t,x,u(t,x))\cdot(u(t,x)-w(t,x))\,dx\,dt -2\delta\int_0^{t_*} \langle \mathcal B u(t), u(t)-w(t)\rangle\,dt.
\end{multline}
Note that the following finite-dimensional inequality holds \be\label{e:monot}\left(\frac {|\xi|^{2p-2}}{\epsilon+|\xi|^{p}}\xi - \frac {|\eta|^{2p-2}}{\epsilon+|\eta|^{p}}\eta \right)\cdot (\xi-\eta)\geq 0, \quad  p\geq 1, \epsilon>0,\ \xi,\eta\in \M.\ee This follows from the convexity of the potential $$\xi\mapsto \frac {|\xi|^p} p - \frac \epsilon p \log (\epsilon + |\xi|^p), \quad \xi\in \M.$$
Applying Cauchy's inequality, inequality \eqref{e:monot} and making some rearrangements, we derive from \eqref{E:wvepap1} that
\begin{multline} \label{E:wvepap2} 
\|u(t_*)-w(t_*)\|^2 +2\int_{Q_{t_*}} \partial_t w(t,x) \cdot u(t,x)\,dx\,dt+2 \int_{Q_{t_*} \cap U} |\nabla u(t,x)|^{p(t,x)}\,dx\,dt\\ \leq \|u_0-w(0)\|^2 +\|
w(t_*)\|^2-\|w(0)\|^2+2
\int_{Q_{t_*} \cap U} \frac {|\nabla u(t,x)|^{2p(t,x)-2}}{\epsilon+|\nabla u(t,x)|^{p(t,x)}}\nabla u(t,x)\cdot\nabla w(t,x) \,dx\,dt\\
+2\int_{Q_{t_*}\backslash U}\frac{|\nabla w(t,x)|^{2p(t,x)-2}}{\epsilon+|\nabla w(t,x)|^{p(t,x)}}\nabla w(t,x) \cdot (\nabla w(t,x)-\nabla u(t,x))\,dx\,dt\\ +2\int_{Q_{t_*}} f(t,x,u(t,x))\cdot(u(t,x)-w(t,x))\,dx\,dt\,\\ +2 \epsilon\int_{Q_{t_*} \cap U} \frac{|\nabla u(t,x)|^{p(t,x)}}{\epsilon+|\nabla u(t,x)|^{p(t,x)}}\,dx\,dt+\frac \delta 2 \int_0^{t_*}\langle \mathcal B w(t), w(t)\rangle\,dt,
\end{multline} and \eqref{E:wvepap} follows. 
\end{proof}

\begin{remark1} In Lemma \ref{L:ap}, the solution $u$ satisfies the a priori bound \begin{equation}\label{e:ab1} \|u\|_{L^{2}(0,T;(H^r(\Omega))^N)}+\|u\|_{H^{1}(0,T;((H^r(\Omega))^*)^N)}+\|u\|_{C([0,T];(L^2(\Omega))^N)}\leq C,\end{equation} which comes from the proof of Lemma \ref{l:l1}. We claim that the constant $C$ here does not depend on $\epsilon$ (but may depend on $\delta$). This follows from the fact that the operator $\mathcal Q_M=\mathcal Q$ satisfies estimate \eqref{e:q2} uniformly in $\epsilon$, see \eqref{e:q2l}. Moreover, from \eqref{e:512} we can derive that  \begin{equation}\label{e:ab2} \delta\|u\|^2_{L^{2}(0,T;(H^r(\Omega))^N)}\leq C,\end{equation} where $C$ is independent of $\epsilon,\delta$. \end{remark1}

\begin{proof}[Proof of Theorem \ref{mthm}] 1. Let $u$ be a solution to \eqref{E:apprp1} in the sense of Lemma \ref{L:ap}. Set $$Z=\frac {|\nabla u(t,x)|^{2p(t,x)-2}}{\epsilon+|\nabla u(t,x)|^{p(t,x)}}\nabla u(t,x).$$ Then \be\label{z1e}\|Z\|_{ L^\infty(Y;\M)}\leq 1,\ee and from \eqref{e:ab1} we derive that \be \label{zeps} \|Z\|_{L^{2}(0,T;(C(\overline\Omega))^N)}\leq C,\ee where $C$ does not depend on $\epsilon$ (but may depend on $\delta$). We will aso need that \begin{multline}\label{zvpv}\|Z\|_{L^{p_+'}(Q;\M)}^{p_+'} \leq \int_Q |\nabla u(t,x)|^{p_+'(p(t,x)-1)}\, dx\,dt\\ \int_Q (1+|\nabla u(t,x)|^{p'(t,x)(p(t,x)-1)})\, dx\,dt=|Q|+ VPV^{p(\cdot)}(u).\end{multline}

2. Since $(u,Z)$ is a solution to \eqref{E:apprp1} in the sense of Lemma \ref{L:ap}, we have \be\label{e:apx1} \left\langle\frac {du} {dt},v \right\rangle+\delta \langle\mathcal B u,v \rangle +\big(Z, \nabla v\big)=(f(\cdot,u),v) \ee a.e. on $(0,T)$,
 \begin{multline} \label{e:apx2} 
\|u(t_*)-w(t_*)\|^2 +2\int_{Q_{t_*}} \partial_t w(t,x) \cdot u(t,x)\,dx\,dt+2\, VPV^{p(\cdot)}(u)({Q_{t_*} \cap U})\\ \leq \|u_0-w(0)\|^2 +\|
w(t_*)\|^2-\|w(0)\|^2+2
\int_{Q_{t_*} \cap U} Z(t,x)\cdot\nabla w(t,x) \,dx\,dt\\
+2\int_{{Q_{t_*}}\backslash U}\frac{|\nabla w(t,x)|^{2p(t,x)-2}}{\epsilon+|\nabla w(t,x)|^{p(t,x)}}\nabla w(t,x) \cdot (\nabla w(t,x)-\nabla u(t,x))\,dx\,dt\\ +2\int_{Q_{t_*}} f(t,x,u(t,x))\cdot(u(t,x)-w(t,x))\,dx\,dt\\ +2\epsilon |{Q_{t_*} \cap U}|+\frac \delta 2 \|w\|_{L^{2}(0,t_*;(H^r(\Omega))^N)}
\end{multline} for all $t_*\in [0,T]$, and
\be \label{e:apx3} u(0)=u_0,\ee for all sufficiently smooth functions $v,w:\overline Q\to \R$ and for every open set $U$, $Y\subset U\subset Q$. 

To prove the theorem, we are going to pass to the limit in \eqref{e:apx1}, \eqref{e:apx2}, \eqref{e:apx3}, first as $\epsilon\to 0$, and then as $\delta\to 0$.

3. Take a sequence of solutions $(u_k,Z_k)$ to \eqref{E:apprp1} with $\epsilon=\epsilon_k\to 0$. Due to \eqref{e:ab1}, \eqref{zeps}, without loss of generality we have \begin{gather}u_k\to u\ \mathrm{weakly}\ \mathrm{in} \ L^{2}(0,T;(H^r(\Omega))^N), \\ u_k\to u\ \mathrm{weakly*}\ \mathrm{in} \ L^{\infty}(0,T;(L^2(\Omega))^N),\\ u_k^\prime\to u^\prime \ \mathrm{weakly}\ \mathrm{in} \ {L^{2}(0,T;((H^r(\Omega))^*)^N)}, \\ Z_k\to Z\ \mathrm{weakly*}\ \mathrm{in} \ L^{2}(0,T;L^\infty(\Omega;\M)).\end{gather} By the Aubin-Lions-Simon lemma \cite{Si87}, \be u_k\to u\ \mathrm{strongly}\ \mathrm{in} \ L^{2}(Q)^N\ \mathrm{and}\ \mathrm{in}\ C([0,T]; ((H^r(\Omega))^*)^N).\ee Hence, by a classical continuity property of Nemytskii operators \cite{Kr64}, \be f(\cdot,u_k)\to f(\cdot,u)\ \mathrm{strongly}\ \mathrm{in} \ L^{2}(Q)^N.\ee Moreover, by \cite[Corollary 6]{Si87} (cf. the proof of Lemma \ref{l:l1}),  \be u_k\to u\ \mathrm{strongly}\ \mathrm{in} \ L^{2p_+}(0,T;C^1(\overline\Omega)^N).\ee Consequently, employing \eqref{e:p-p}, \be \nabla u_k\to \nabla u\ \mathrm{strongly}\ \mathrm{in} \ L^{p(\cdot)}(Q;\M),\ee and \be VPV^{p(\cdot)}(u_k) (V)\to VPV^{p(\cdot)}(u)(V) \label{vtvlim} \ee for any Borel set $V\subset Q$. We can thus pass to the limit as $k\to \infty$ in \eqref{e:apx1},\eqref{e:apx3} and in the linear (w.r.t. $u$) and constant (w.r.t. $k$) members of \eqref{e:apx2}, and in the third term in \eqref{e:apx2}. We can also pass to the limit inferior, keeping the sign of the inequality, in the  first member of \eqref{e:apx2}, since this term is lower semicontinuous, see the $\delta$-limit below in the proof for the details of a similar reasoning. Finally, by the Lebesgue dominated convergence theorem, $$\frac{|\nabla w|^{2p(\cdot)-2}}{\epsilon_k+|\nabla w|^{p(\cdot)}}\nabla w \to |\nabla w|^{p(\cdot)-2}\nabla w$$ in $L^2(Q\backslash U)^N$. Hence, passing to the limit in \eqref{e:apx2} we get  \begin{multline} \label{e:apx2.1} 
\|u(t_*)-w(t_*)\|^2 +2\int_{Q_{t_*}} \partial_t w(t,x) \cdot u(t,x)\,dx\,dt+2\, VPV^{p(\cdot)}(u)({Q_{t_*} \cap U})\\ \leq \|u_0-w(0)\|^2 +\|
w(t_*)\|^2-\|w(0)\|^2+2
\int_{Q_{t_*} \cap U} Z(t,x)\cdot\nabla w(t,x) \,dx\,dt\\
+2\int_{{Q_{t_*}}\backslash U}|\nabla w(t,x)|^{p(t,x)-2}\nabla w(t,x) \cdot (\nabla w(t,x)-\nabla u(t,x))\,dx\,dt\\ +2\int_{Q_{t_*}} f(t,x,u(t,x))\cdot(u(t,x)-w(t,x))\,dx\,dt+\frac \delta 2 \|w\|_{L^{2}(0,t_*;(H^r(\Omega))^N)}
\end{multline} for all $t_*\in [0,T]$. Note that  \eqref{e:ab2} and \eqref{z1e} still hold for the limit. 

4. Taking $U=Q$, $w\equiv 0$ in \eqref{e:apx2.1}, and remembering \eqref{e:ab2}, we derive \be \label{e:apx4} 
\|u(t_*)\|^2 +2\, VPV^{p(\cdot)}(u)({Q_{t_*}}) \leq C(1+\delta) +\int_0^{t_*} \|u(t)\|^2 \,dt 
\ee for all $t_*\in [0,T]$, whence by Gr\"onwall's lemma \be \label{e:apx4} 
\|u(t_*)\|^2 + VPV^{p(\cdot)}(u) \leq C, 
\ee  where $C$ does not depend on $\delta$ (we assume that the possible range of $\delta$ is bounded). Then, by \eqref{zvpv} and \eqref{vtvlim}, \be\label{zbound}\|Z\|_{L^{p_+'}(Q;\M)}\leq C. \ee Moreover, \eqref{e:apx4} implies 
\be \label{l1bound} \int_Q |u(t,x)|+|\nabla u(t,x)|\,dx\,dt\leq C+\int_Q |\nabla u(t,x)|^{p(t,x)}\,dx\,dt\leq C.\ee Observe also that since $H^r\subset W^{1,p_+}$, relations \eqref{e:apx1}, \eqref{e:ab2}, \eqref{e:apx4}, \eqref{fid}  yield \be\label{e:bde} \|u'\|_{L^{\min(p_+',2)}(0,T;((H^r(\Omega))^*)^N)}\leq C.\ee
All the constants $C$ here are $\delta$-independent. 

5. As a result of the previous steps, we can find a sequence of pairs  $(u_m,Z_m)$ satisfying \eqref{e:apx1},\eqref{e:apx3},\eqref{e:apx2.1} with $\delta=\delta_m\to 0$.  Due to the $\delta$-independent estimates \eqref{e:apx4}, \eqref{zbound}, \eqref{e:bde}, without loss of generality we may assume that \begin{gather}\label{sh1} u_m\to u\ \mathrm{weakly*}\ \mathrm{in} \ L^{\infty}(0,T;(L^2(\Omega))^N) \ \mathrm{and}\ \mathrm{weakly}\ \mathrm{in}\ L^2(Q)^N,\\ u_m^\prime\to u^\prime \ \mathrm{weakly}\ \mathrm{in} \ {L^{{\min(p_+',2)}}(0,T;((H^r(\Omega))^*)^N)}, \\ Z_m\to Z\ \mathrm{weakly}\ \mathrm{in} \ L^{p_+'}(Q;\M).\end{gather} 
By the Aubin-Lions-Simon lemma \cite{Si87} and bounds \eqref{l1bound}, \eqref{e:bde}, \be u_m\to u\ \mathrm{strongly}\ \mathrm{in} \ L^{1}(Q)^N,\ee and thus without loss of generality \be \label{sh2} u_m\to u\ \mathrm{a.e.}\ \mathrm{in} \ Q.\ee 

As above, we can pass to the limit in \eqref{e:apx1},\eqref{e:apx3} and in the linear and constant members of \eqref{e:apx2.1}. Note that the second member in \eqref{e:apx1} goes to zero by virtue of \eqref{e:ab2}, and the last member of \eqref{e:apx2.1} also goes to zero. For a.a. $t_*\in(0,T)$, we can pass to the limit inferior, keeping the sign of the inequality, in the first member of \eqref{e:apx2.1} (we test \eqref{e:apx2.1} by smooth compactly supported nonnegative scalar functions of time, and pass to the limit in the sense of scalar distributions, see \cite{Li96, Vo11,SVU15} for technical details in similar situations). By Theorem \ref{L:twinp}, $u \in BVPV^{p(\cdot)}$ and we can  pass to the limit inferior in the $VPV^{p(\cdot)}$-term. 

It remains to pass to the limit in the penultimate term of \eqref{e:apx2.1} involving the fidelity. By \cite[Theorem 2.6]{mt5}, \eqref{sh1} and \eqref{sh2} yield \be \label{sh3} f(\cdot,u_m)\to f(\cdot,u)\ \mathrm{weakly}\ \mathrm{in} \ L^{2}(Q)^N.\ee Consider the functional $$F: L^2(Q)^N\to\R, \ F(u)=\int_{Q_{t_*}}  f(t,x,u(t,x))\cdot u(t,x)\,dx\,dt.$$ A classical result on Nemytskii operators \cite{Kr64,mt5} implies that $F$ is bounded and continuous. Since the function $v\mapsto v\cdot f(t,x,v)$ is concave, $F$ is concave and therefore weakly upper semicontinuous. This fact together with \eqref{sh1} and \eqref{sh3} gives the opportunity to pass to the limit superior in the penultimate term of \eqref{e:apx2.1}, maintaining the sign of the inequality.

Hence, in the limit we get \begin{multline} \label{e:apx10} 
\|u(t_*)-w(t_*)\|^2 +2\int_{Q_{t_*}} \partial_t w(t,x) \cdot u(t,x)\,dx\,dt+2\, VPV^{p(\cdot)}(u)({Q_{t_*} \cap U})\\ \leq \|u_0-w(0)\|^2 +\|
w(t_*)\|^2-\|w(0)\|^2+2
\int_{Q_{t_*} \cap U} Z(t,x)\cdot\nabla w(t,x) \,dx\,dt\\
+2\int_{{Q_{t_*}}\backslash U}|\nabla w(t,x)|^{p(t,x)-2}\nabla w(t,x) \cdot (\nabla w(t,x)-\nabla u(t,x))\,dx\,dt\\ +2\int_{Q_{t_*}} f(t,x,u(t,x))\cdot(u(t,x)-w(t,x))\,dx\,dt
\end{multline} for a.a. $t_*\in(0,T)$. By a Lions-Magenes lemma \cite[Lemma 2.2.6]{ZV08},  $$u\in C_w(0,T; L^2(\Omega)^N).$$ Therefore, the first term of \eqref{e:apx10} is lower semicontinuous in time. Since $VPV^{p(\cdot)}(u)$ is a Radon measure on $Q$, the third term is also lower semicontinuous in time.  The remaining members are continuous in time. Hence, \eqref{e:apx10} holds for all $t_*\in[0,T]$, in particular, for $t_*=T$. Thus, $(u,Z)$ satisfies \eqref{E:zp}, \eqref{E:wvep}, \eqref{E:ic}. By density, the test functions $v$ and $w$ in \eqref{E:zp} and \eqref{E:wvep} can be taken from the spaces indicated in Definition \ref{D:wwdefpx}. 
\end{proof}
\subsection{Semigroup solutions} \label{s:ssol} In this section we assume that $p$ does not depend on time, i.e. we have a $Y$-semicontinuous variable exponent $p :\Omega \to [1,p_+]$, and $Y=[p(x)=1]$ is the critical set in $\Omega$. To fix the ideas and avoid unnecessary techicalities, we also assume that $f\equiv 0$. We are thus left with the following Neumann problem: \be \label{E:mainpl0} \left\{
\begin{array}{ll}
\displaystyle \partial_t u(t,x)=\divo Z(t,x),\\
Z(t,x)=|\nabla u(t,x)|^{p(x)-2}\nabla u(t,x),\\ Z(t,x)\nu(x) =0,\ x\in \partial \Omega,
\\ u(0,x)=u_0(x). \end{array}
\right. \ee 
Consider the functional $\Psi:L^2(\Omega)^N\to \R$, 
\be \label{E:psi} \Psi(u)=\left\{
\begin{array}{ll}
TV(u)(Y) + \int_{\Omega\setminus Y} \frac {\abs{\nabla u}^{p(x)}}{p(x)}\,dx, & u\in (BV^{p(\cdot)}(\Omega))^N\cap L^2(\Omega)^N,\\ +\infty & \mathrm{otherwise}
. \end{array} \right.  \ee

Getting rid of the auxiliary variable $Z$, it is easy to observe that \eqref{E:mainpl0} is a formal gradient flow of $\Psi$ in $L^2(\Omega)^N$: $$u^\prime+\mathrm{grad}_{L^2(\Omega)^N}\, \Psi(u)=0.$$ However, the functional $\Psi$ is not regular enough to employ this fact immediately. But we still have   
\begin{lemma} \label{l:lsc1} The functional $\Psi$ is proper, convex, lower-semicontinuous and has a dense domain.
\end{lemma}

\begin{proof} The lower-semicontinuity can be checked in a way similar to Theorem \ref{L:twinp} (see also \cite{HHLT13}). The rest is obvious. \end{proof}

Lemma \ref{l:lsc1} implies that the subdifferential $\partial \Psi$ has a dense domain $\mathrm{dom}\, \partial\Psi$, see \cite[Subsection 17.2.2]{ABM14}.  For convex lower semicontinuous functionals, the notion of subdifferential is the correct extension of the notion of gradient, see again \cite[Subsection 17.2.2]{ABM14}. We can thus give the following definition, which is also motivated by \cite{va04, ABCM01} and by \cite{AM11,AM13}. \begin{definition} A function $$W^{1,\infty}_{loc}((0,+\infty); L^2(\Omega)^N)\cap C([0,+\infty); L^2(\Omega)^N)$$ is called a semigroup solution to \eqref{E:mainpl0} provided $$u(t)\in \mathrm{dom}\, \partial\Psi$$ for all $t\geq 0$, \be 0\in u^\prime+\partial\Psi(u).\ee for a.a. $t> 0$, and \be u(0)=u_0.\ee \end{definition}

By \cite[Theorem 17.2.3]{ABM14}, we shortly have

\begin{theorem} For any $u_0\in L^2(\Omega)^N$, there is a unique semigroup solution to \eqref{E:mainpl0}.\end{theorem}

\begin{remark} Semigroup solutions automatically have plenty of nice properties provided by the gradient flow structure, such as continuous dependence on $u_0$ and long-time convergence to the infimum of $\Psi$, we refer to \cite[Section 17]{ABM14} for the details.  \end{remark}

\begin{remark1}\label{r:ssol} Let $(u,Z)$ be a sufficiently regular weak solution to \eqref{E:mainpl0} in the sense of Definition \ref{D:wwdefpx}. We claim that $u$ is a (unique) semigroup solution. Indeed, fix an element $h\in (BV^{p(\cdot)}(\Omega))^N\cap L^2(\Omega)^N$. For any $s\geq 1$, let us introduce the function $$\psi_{s}: \M\to \R,\, \psi_{s}(A)= \frac{|A|^{s}}{s}.$$ Then we can rewrite \eqref{E:wve5psd} and \eqref{E:wve6p} in the subdifferential form \be  \label{E:wve7psd}  Z(t,x)\in \partial \psi_{p(x)}(\nabla u(t,x))\ee a.e. in $Q$. In particular, for a.a. $x\in \Omega\backslash Y$ we have 
\be  \label{E:wve8psd}  \psi_{p(x)}(\nabla (u+h)(t,x))\geq \psi_{p(x)}(\nabla u (t,x)) +\nabla h(x) \cdot Z(t,x), \ee 
whence \be  \label{E:wve9psd} \int_{\Omega\setminus Y} \frac {\abs{\nabla (u+h)}^{p(x)}}{p(x)}\,dx\geq \int_{\Omega\setminus Y} \frac {\abs{\nabla u}^{p(x)}}{p(x)}\,dx + \int_{\Omega\setminus Y} \nabla h(x) \cdot Z(t,x)\,dx\ee By \cite[Corollary C.7]{va04} we have \be
\label{E:wve10psd} TV(u+h)(Y)\geq \frac {\int_{Y}  Z(t,x)\,d\,\nabla (u+h)}{\|Z\|_{L^\infty(U;\M)}} \ee for any open set $U\subset \Omega$ containing $Y$. Since  $Z\cdot \nabla u=|\nabla u|$ a.e. in $Y$ and $\|Z\|_{L^\infty(Y;\M)}\leq 1$, provided $Z$ is continuous we infer \begin{multline}
\label{E:wve11psd} TV(u+h)(Y)\geq \int_{Y}  Z(t,x)\,d\,\nabla h(x)+\int_{Y} \nabla u(t,x)\cdot Z(t,x)\,dx\\=\int_{Y}  Z(t,x)\,d\,\nabla h(x)+\int_{Y} |\nabla u(t,x)|\,dx.\end{multline} Adding \eqref{E:wve9psd} and \eqref{E:wve11psd}, remembering that $\partial_t u(t,x)=\divo Z(t,x)$ and $Z\nu=0$ \ $\mathcal{H}^{n-1}$-a.e. in $\partial \Omega$, and using \cite[Theorem C.9]{va04}, we deduce \be
\label{E:wve12psd} \Psi(u+h)\geq \Psi(u)- \int_{\Omega}  \divo Z(t,x)\cdot h(x)\,dx=\Psi(u)- \int_{\Omega}  \partial_t u(t,x)\cdot h(x)\,dx.\ee Trivially, \eqref{E:wve12psd} also holds for $h\in L^2(\Omega)^N\setminus(BV^{p(\cdot)}(\Omega))^N$. Hence, \be -u^\prime\in \partial\Psi(u),\ee and $u$ is a semigroup solution.  \end{remark1}


\subsection*{Acknowledgement} DV thanks Goro Akagi, 
Stanislav Antontsev and Peter H\"{a}st\"{o} for useful correspondence. 

\bibliographystyle{plain}      
\bibliography{vtv}   
\end{document}